\documentclass[12pt,reqno]{amsart}
\usepackage{amsmath,amsthm,amssymb,mathrsfs,stmaryrd,color}
\usepackage[all]{xy}
\usepackage{url}
\usepackage{slashed}
\usepackage{geometry}

\usepackage[utf8]{inputenc}
\usepackage[T1]{fontenc}

\usepackage{relsize}
\usepackage[bbgreekl]{mathbbol}
\usepackage{amsfonts}

\usepackage{tikz-cd}

\usepackage[mathscr]{eucal}
\usepackage{eucal}
\usepackage{url}

\usepackage{setspace}
\usepackage{subcaption}
\usepackage{array}
\usepackage{wrapfig}
\usepackage{multirow}
\usepackage{tabularx}
\usepackage{pst-node}
\usepackage{mathtools}
\usepackage{bbm}

\usepackage{enumitem}
\setlist[enumerate]{itemsep=2pt,parsep=2pt,before={\parskip=2pt}}

\usepackage[colorlinks=true,hyperindex, linkcolor=magenta, pagebackref=false, citecolor=cyan,pdfpagelabels]{hyperref}

\newcommand{\cosimp}[3]{\xymatrix@R=50pt@C=50pt@1{#1 \ar@<.4ex>[r] \ar@<-.4ex>[r] & {\ }#2 \ar@<0.8ex>[r] \ar[r] \ar@<-.8ex>[r] & {\ } #3 \ar@<1.2ex>[r] \ar@<.4ex>[r] \ar@<-.4ex>[r] \ar@<-1.2ex>[r] & \cdots }}

\newcommand{\adjunction}[4]{\xymatrix@R=50pt@C=50pt@1{#1{\ } \ar@<0.3ex>[r]^-{ {\scriptstyle #2}} & {\ } #3 \ar@<0.3ex>[l]^{ {\scriptstyle #4}}}}

\newcommand{\defeq}{\vcentcolon=}

\newtheorem{theorem}{Theorem}[subsection]
\newtheorem*{theorem*}{Theorem}
\newtheorem*{definition*}{Definition}
\newtheorem{proposition}[theorem]{Proposition}
\newtheorem{lemma}[theorem]{Lemma}
\newtheorem{corollary}[theorem]{Corollary}
\newtheorem{conjecture}[theorem]{Conjecture}

\theoremstyle{definition}

\newtheorem{definition}[theorem]{Definition}

\newtheorem{remark}[theorem]{Remark}

\newtheorem{observation}[theorem]{Observation}
\newtheorem{example}[theorem]{Example}

\newenvironment{hproof}{%
  \proof}{\endproof}

\DeclareMathOperator{\mmod}{mod}

\DeclareMathOperator{\End}{End}

\newcommand{\Div}{\operatorname{Div}}
\newcommand{\Br}{\operatorname{Br}}
\newcommand{\Hom}{\operatorname{Hom}}

\newcommand{\Spec}{\operatorname{Spec}}
\newcommand{\Aut}{\operatorname{Aut}}

\newcommand{\GL}{\operatorname{GL}}
\newcommand{\PGL}{\operatorname{PGL}}

\newcommand{\Id}{\operatorname{Id}}
\newcommand{\rk}{\operatorname{rk}}

\newcommand{\Pic}{\operatorname{Pic}}
\newcommand{\Char}{\operatorname{char}}
\newcommand{\Gal}{\operatorname{Gal}}
\newcommand{\Sym}{\operatorname{Sym}}
\newcommand{\Gr}{\operatorname{Gr}}

\newcommand{\Bl}{\operatorname{Bl}}
\newcommand{\codim}{\operatorname{codim}}

\newcommand{\Gm}{\mathbb{G}_{\mathrm{m}}}

\def \ZZ {\mathbb{Z}}

\def \CC {\mathbb{C}}
\def \AA {\mathbb{A}}

\def \kk {\Bbbk{}}
\def \LL {\mathbb{L}}
\def \FF {\mathbb{F}}
\def \PP {\mathbb{P}}
\def \ge {\geqslant}
\def \le {\leqslant}
\def \et {{\acute e}t}

\def \kk {\Bbbk{}}

\newcommand{\cE}{\mathcal{E}}
\newcommand{\cF}{\mathcal{F}}

\newcommand{\cL}{\mathcal{L}}

\newcommand{\cO}{\mathcal{O}}
\newcommand{\cA}{\mathcal{A}}
\newcommand{\fm}{\mathfrak{m}}

\begin{document}

\title{Automorphisms of symmetric powers and motivic zeta functions}
\author{Vladimir Shein}
\address{National Research University Higher School of Economics, Russian Federation}
\email{vladimir.sh2000@gmail.com}

\begin{abstract}

We prove that if $X$ is a smooth projective variety of dimension greater than 1 over a field $K$ of characteristic zero such that $\operatorname{Pic}(X_{\bar{K}}) = \mathbb{Z}$ and $X_{\bar{K}}$ is simply connected, then the natural map $\rho: \operatorname{Aut}(X) \to \operatorname{Aut}(\operatorname{Sym}^d(X))$ is an isomorphism for every $d > 0$.

We also partially compute the motivic zeta function of a Severi-Brauer surface and explain some relations between the classes of Severi-Brauer varieties in the Grothendieck ring of varieties.

\end{abstract}

\maketitle


\section{Introduction}\label{section:introduction}

Let $X$ be a quasi-projective variety over a field $K$. The Kapranov motivic zeta function of $X$ is a formal power series
\[
Z_{\text{mot}}(X, t) = \sum_{n \ge 0} [\Sym^n(X)]t^n,
\]
where $[\Sym^{n}(X)]$ are the classes of the symmetric powers of $X$ in the Grothendieck ring \linebreak $K_0(\text{Var}/K)$ of varieties over $K$. Kapranov conjectured  that $Z_{\text{mot}}(X, t)$ may be a rational function (as a power series in $K_0(\text{Var}/K)$) for arbitrary variety $X$ ~\mbox{\cite[Remark~1.3.5]{Kapranov}}. If so, this would give a completely new proof of the first Weil conjecture. Kapranov showed that it is the case when $X$ is a curve with a rational point, and then D. Litt  proved the rationality for arbitrary geometrically connected curve (see ~\mbox{\cite[Theorem~10]{Litt}}). However, it turned out that $Z_{\text{mot}}(X, t)$ is not always rational even for complex surfaces. It is a theorem of Larsen and Lunts ~\mbox{\cite[Theorem~1.1]{Lunts2}} that for a smooth complex surface $X$ zeta function $Z_{\text{mot}}(X, t)$ is rational if and only if $X$ has negative Kodaira dimension.

The problem of finding reasonable criteria for the rationality of motivic zeta function remains open for complex varieties of dimension greater than 2, and even for surfaces over non-algebraically closed fields. The present paper grew up from the attempt to prove the rationality of motivic zeta function of Severi-Brauer varieties. Although we do not prove it in general, we give some partial results, together with the computation of the automorphism group of symmetric powers of some classes of projective varieties, which may have independent interest.

The paper consists of two parts. Section \ref{section:automorphisms} is devoted to the problem of finding the automorphisms group of symmetric powers of a variety in terms of the automorphisms group of the variety itself. It is known that if $X$ is a smooth projective curve of genus $g$, with $g > 2$, over an algebraically closed field of characteristic different from two, then the natural morphism $\rho: \operatorname{Aut}(X) \to \operatorname{Aut}(\operatorname{Sym}^d(X))$ is an isomorphism for $d > 2g - 2$ (see \cite[Theorem~1.1]{BG}). Theorem \ref{theorem:automorphisms} is 
the analogue of this result in higher dimensions for some classes of projective simply connected varieties.

In the subsequent sections, we study motivic zeta functions of Severi-Brauer varieties. We note in proposition \ref{ratinoalinquotient} that it is rational in the quotient ring $K_0(\text{Var}/K)/(\LL)$. However, it is much harder to prove the rationality in $K_0(\text{Var}/K)$ itself. For a Severi-Brauer variety $B$, we compute the classes $[\Sym^{r}(B)]$ in the completed Grothendieck ring $\varprojlim K_0(\text{Var}/K)/(\LL^n)$ for all $r$ coprime with $\dim{B} + 1$. Theorem  \ref{zetaforsurface} gives a partial expression for motivic zeta function of a Severi-Brauer surface. Before proving it, we study the classes of the symmetric powers of Severi-Brauer varieties in $K_0(\text{Var}/K)$. The key idea is explained in proposition \ref{highsym}. We find it convenient to use the language of twisted sheaves, and give a largely self-contained review about them in section \ref{prelim}.

The author would like to thank his scientific advisor Vadim Vologodsky for stating the problem, fruitful discussions and constant attention to this work. I am also grateful to Constantin Shramov, Aleksei Piskunov and Igor Spiridonov for useful comments.




\section{Preliminaries} \label{prelim}

\subsection{Preliminaries on symmetric powers}

Let $K$ be a field and let $X$ be a quasi-projective variety over $K$. Then $X^d$ is again quasi-projective, and the quotient $X^d / S_d$ by the action of the symmetric group $S_d$ on $X^d$ is well-defined as a variety. This is the \textit{$d$-fold symmetric power} $\Sym^d(X)$. We adopt the convention that $\Sym^0(X) = \Spec{K}$.

\begin{remark} \label{bijection}
Note that, since quotients by finite groups commute with flat extensions, it follows that $\Sym^d(X_{S}) \cong \Sym^d(X)_{S}$ for every scheme $S$ over $K$. In particular, we have a bijection of sets: $\Sym^d(X)(\bar{K}) = \{z \in \Div^{\text{eff}}(X_{\bar{K}}) \mid \deg{z} = d  \}$.

\end{remark}

Symmetric powers have a natural decomposition into locally closed subsets. Let $(X^{r})^{\circ}$ be the complement in $X^r$ of the union of the big diagonals in $X^r$. Given positive integers $d_1 \le d_2 \le \cdots \le d_k$ such that $d_1 + \cdots + d_k = n$, there is a locally closed embedding $(X^{r})^{\circ} \hookrightarrow X^d$ given by $\Delta_{d_1} \times \ldots \times \Delta_{d_r}$, where $\Delta_{i}: X \to X^{i}$ is the diagonal embedding. We denote by $\widetilde{X}_{d_1, \ldots, d_r}$ the image of $(X^{r})^{\circ}$, and by $X_{d_1, \ldots, d_r}$ the image of $\widetilde{X}_{d_1, \ldots, d_r}$ in $\Sym^d(X)$ under the natural projection $\pi: X^d \to \Sym^d(X)$. It is easily seen that the $X_{d_1, \ldots, d_r}$ give a partition of $\Sym^d(X)$ into locally closed subsets and that $X_{d_1, \ldots, d_r}(\bar{K})=\{z \in \Div^{\text{eff}}(X_{\bar{K}}) \mid z = \sum_{i=1}^{r} d_i x_i \text{ for some } x_1, \ldots, x_r \in X_{\bar{K}} \}$.

One can easily verify that any symmetric power of a smooth curve is smooth. However, if $\dim{X} > 1$, then $\Sym^{d}(X)$ is always singular for $d > 1$, and $\widetilde{X}_{1, \ldots, 1}$ is the smooth locus of $\Sym^{d}(X)$.

We will also need some facts about multisymmetric polynomials. Let $R$ be a ring. Consider the ring $R[X_1, \ldots, X_d]$, where $X_i$'s are vectors $X_i = (x_{i,1}, \ldots, x_{i, n})$ of $n$ independent variables. There is an action of the symmetric group $S_d$ on $R[X_1, \ldots, X_d]: \sigma(f(X_1, \ldots, X_d)) = f(X_{\sigma^{-1}(1)}, \ldots, X_{\sigma^{-1}(d)})$. The invariant subring $R[X_1, \ldots, X_d]^{S_d}$ is called \textit{the ring of multisymmetric polynomials in $d$-vector variables}. This is the coordinate ring of $\Sym^d(\AA_R^{n})$.

\begin{definition}

The elementary symmetric polynomials $e_{k_1, \ldots, k_n}(X_1, \ldots, X_d)$ of $d$ vector variables are defined by the formula
\[
\prod_{i=1}^{d}(1+x_{i,1}t_1+\cdots+x_{i,n}t_n) = 1+ \sum_{k_1, \ldots, k_n} e_{k_1, \ldots, k_n}(X_1, \ldots, X_d) t_1^{k_1} \cdots t_n^{k_n},
\]
where the sum is over all $k_1, \ldots, k_n$ with $1 \le \sum_{i=1}^{n} k_i \le d$.

\end{definition}

\begin{theorem}[{\cite[Chapter 4. 2.4]{Discriminants}}] \label{sym_poly}

Let $R$ be an algebra over a field $K$ of characteristic zero. The ring $R[X_1, \ldots, X_d]^{S_d}$ is generated by the elementary symmetric polynomials $e_{k_1, \ldots, k_n}(X_1, \ldots, X_d)$.

\end{theorem}

\begin{remark}

The cited book proves the theorem for $R = \CC$ only, but it is easy to see that the proof works over any algebra over a field of characteristic zero. However, if $\Char{k} > 0$ and $d, n$ are large enough, then the theorem is no longer true. See, for example, \cite{Neeman}.

\end{remark}

\subsection{The Kapranov motivic zeta function.}

Let $K$ be any field and let $S$ be a variety over $K$. \textit{The Grothendieck ring} $K_0(\text{Var}/S)$ \textit{of varieties over} $S$ is the quotient of the free abelian group generated by the isomorphism classes $[X]_S$ of varieties over $S$, by the relations
\[
[X]_S = [Z]_S + [X \, \backslash \, Z]_S,
\]
where $Z$ is a closed subvariety of $X$. It is easily seen (by induction on $\dim{X}$) that if $X = Z_1 \sqcup \ldots \sqcup Z_m$, where all $Z_i$ are locally closed subvarieties of $X$, then $[X]_S = [Z_1]_S+ \ldots + [Z_m]_S$. The product in $K_0(\text{Var}/S)$ is given by
\[
[X]_S \cdot [Y]_S = [(X \times_{S} Y)_{\text{red}}]_S
\]
and extended by linearity. It is clear that this product endows $K_0(\text{Var}/S)$ with the structure of a commutative ring, and that $[S]$ is a unit in $K_0(\text{Var}/S)$.

If $S = \Spec{K}$, we simply write $[X]$ instead of $[X]_{\Spec{K}}$.

We denote by $\LL_S$ the class of $\AA_{S}^1$ in $K_0(\text{Var}/S)$ and simply by $\LL$ the class of $\AA_{K}^1$ in $K_0(\text{Var}/K)$. Note that $(\AA_{S}^1)^{n} = \AA_{S}^n$, so $[\AA_{S}^n] = \LL_S^n$. Using the decomposition $\PP_S^{n} = \PP_S^{n-1} \sqcup \AA_{S}^n$, we obtain by induction that $[\PP^{n}_S] = 1 + \LL_S + \ldots + \LL_S^n.$

The motivic zeta function $Z_{\text{mot}}(X, t)$ of a variety $X$ was introduced by M. Kapranov in ~\mbox{\cite[1.3]{Kapranov}}. For simplicity, we shall assume $X$ is a quasi-projective variety over $K$.

\begin{definition}[Kapranov motivic zeta function] Let $X$ be a quasi-projective variety over $K$. The motivic zeta function $Z_{\text{mot}}(X, t)$ is defined by the formula:
\[
Z_{\text{mot}}(X, t) = \sum_{n \ge 0} [\Sym^n(X)]t^n \in 1 + K_0(\text{Var}/K)[[t]].
\]

\end{definition}




\begin{remark}

If $\Char{K} > 0$, one usually consider the ring $\tilde{K}_{0}(\text{Var}/K)$. It is defined as the quotient of $K_0(\text{Var}/K)$ by the relations $[X]-[Y]$, where $X$ and $Y$ are such that there exists  a radical surjective morphism $X \to Y$. It is easy to see that there is a well-defined homomorphism of rings $K_0(\text{Var}/K) \to \tilde{K}_{0}(\text{Var}/K)$. If $\Char{K} = 0$, then it is an isomorphism, because in this case any radical surjective morphism is a piecewise isomorphism, whence $[X] = [Y]$.

For simplicity, from now on we write $K_0(\text{Var}/K)$ instead of $\tilde{K}_{0}(\text{Var}/K)$.

\end{remark}

\begin{proposition}[{\cite[Proposition 7.28]{Zeta}}]\label{Zeta1} \label{zeta1}

The map $K_0(\text{Var}/K) \to (1 + K_0(\text{Var}/K)[[t]], \cdot)$, $[X] \to Z_{\text{mot}}(X, t)$ is a group homomorphism. 

\end{proposition}

\begin{proposition}[{\cite[Proposition 7.32]{Zeta}}] \label{zeta2}

$Z_{\text{mot}}(X \times \AA_k^n, t) = Z_{\text{mot}}(X, \LL^n t)$ in $K_{0}(\text{Var}/K)$.

\end{proposition}

\begin{example}
By the previous two propositions, $Z_{\text{mot}}(\AA_k^n, t) = Z_{\text{mot}}(\Spec{K}, {\LL}^n t) = 1/(1-{\LL}^n  t)$ and
\begin{equation} \label{P^n}
Z_{\text{mot}}(\PP_k^n, t) = Z_{\text{mot}}(\Spec{K}, t) \cdot \ldots \cdot Z_{\text{mot}}(\AA_k^n, t) = \frac{1}{(1-t)(1-{\LL}t) \cdots (1-{\LL}^nt)}.
\end{equation}
\end{example}

\begin{remark}

Let $K = \FF_q$ be a finite field. Then there is a natural homomorphism $K_0(\text{Var}/K) \to \ZZ$, taking $[X]$ to $|X(\FF_q)|$. It induces a map $\tilde{K}_{0}(\text{Var}/K)[[t]] \to \ZZ[[t]]$, which takes $Z_{\text{mot}}(X, t)$ to the Hasse-Weil zeta function $Z(X, t)$. Indeed, we have
\[
Z(X, t) = \prod_{x \in X_{\text{cl}}} \frac{1}{1-t^{\deg(x)}}=\prod_{x \in X_{\text{cl}}} (1+t^{\deg(x)}+t^{2\deg(x)}+\ldots)= \sum_{\alpha \in \Div^{\text{eff}}(X)} t^{\deg(\alpha)}.
\]
On the other hand, by remark \ref{bijection}
\[
\aligned
\Sym^{n}(X)(k) &= \{\Sym^{n}(X)(\bar{k})\}^{\Gal(\bar{k}/k)} = \{z \in \Div^{\text{eff}}(X_{\bar{k}}) \mid \deg{z} = n  \}^{\Gal(\bar{k}/k)} = \\ & = \{z \in \Div^{\text{eff}}(X) \mid \deg{z} = n \},
\endaligned
\]
hence $Z(X, t) = \sum_{n \ge 0} |\Sym^n(X)(k)| t^n$.
\end{remark}

The motivic zeta function of $X$ is called \textit{rational} if $Z_{\text{mot}}(X, t) = f(t)/g(t)$ for some polynomials $f,g \in K_0(\text{Var}/K)[t]$, where $g$ is invertible in $K_0(\text{Var}/K)[[t]]$. 

\subsection{Some classes in the Grothendieck ring.}

A morphism $f: X \to Y$ is called \textit{piecewise trivial fibration} with fiber $F$ if there exists a decomposition $Y = Y_1 \sqcup \ldots \sqcup Y_n$, such that $Y_i$ are locally closed subsets of $Y$ and $(X \times_Y Y_i)_{\text{red}} \cong (F \times_{\Spec{K}} Y_i)_{\text{red}}$ for all $i$. It is easy to show that in this case $[X] = [Y] \cdot [F]$ in $K_0(\text{Var}/K)$. In particular, for every vector bundle $E$ on $Y$ (resp. projective bundle $P$ on $Y$), of constant rank $d$, one has $[E] = [Y] \, \LL^d $ (resp. $[P] =  [Y] \cdot [\PP^d] = [Y](1+\LL+\cdots+\LL^d)$).

\begin{example}

Let $X$ be a smooth variety, and let $Y \subset X$ be a smooth closed subvariety of codimension $d$. Consider the blow-up $\Bl_Y(X)$ of $X$ along $Y$. Then the exceptional divisor $E$ is a projective bundle of constant rank $d-1$ over $Y$. Since $\Bl_Y(X) \, \backslash \, E \cong  X   \, \backslash \, Y$, it follows that
\[
[\Bl_Y(X)] = [X] + [Y]([\PP^{d-1}] - 1).
\]

\end{example}

While the Grothendieck ring of varieties is quite poorly understood, there is a natural  description for the quotient ring $K_0(\text{Var}/k)/(\LL)$.

Recall that irreducible varieties $X, Y$ are called \textit{stably birational} if $X \times \PP^{n}$ is birational to $Y \times \PP^{m}$ for some $n, m \ge 0$. Let $\text{SB}/K$ be the set of stably birational equivalence classes $\langle X \rangle$ of irreducible varieties over $K$. It is a commutative semigroup, with multiplication induced by the rule $\langle X \rangle \times \langle Y \rangle = \langle X \times_{\Spec{K}} Y \rangle$, and with the identity element $\langle \Spec{K} \rangle$. Let $\ZZ[\text{SB}/k]$ be the ring associated to the semigroup $\text{SB}/K$.

\begin{theorem}[{\cite{Lar_Lun}}] \label{Lar_Lun}

Let $K$ be a field of characteristic zero. There is a ring isomorphism:
\[
\Phi: K_0(\text{Var}/K)/(\LL) \to \ZZ[\text{SB}/K].
\]

\end{theorem}

\begin{remark}

The original proof is over $K = \CC$, but it actually works over any field of characteristic zero. The main ingredient of the proof is the birational factorization theorem; see, for example, \cite[Theorem 0.3.1, remark 2]{AKMW}.

\end{remark}

\subsection{Twisted sheaves and Severi-Brauer schemes}

The notion of twisted sheaves gives a convenient way to study Severi-Brauer schemes over a general base. In this section we recall some basic facts about them. Our main reference is \cite[I.1-2]{caldararu}.

For an \'etale covering $\mathfrak{U} = (U_i \to X)$ of $X$ we put $U_{i_0 \ldots i_r} = U_{i_0} \times_{X} \cdots \times_{X} U_{i_r}$ and denote by $p^{ij}_{i}: U_{ij} \to U_i$ and $p^{ijk}_{ij}: U_{ijk} \to U_{ij}$ the projections on the $i$-th factor and the $(i, j)$-th factor, respectively (and similarly for other projections).

\begin{definition}

Let $(X, \cO_X)$ be a scheme and let $\alpha = (\alpha_{ijk}) \in \check{C}^2(X, \cO^{\ast}_X)$ be a \v{C}ech $2$-cocycle for an \'etale covering $\mathfrak{U} = (U_i \to X)_{i \in I}$. An $\alpha$\textit{-twisted sheaf} $(\cE, \varphi)$  is a pair $(\{\cE_i\}_{i \in I}, \{\varphi_{ij}  \}_{i, j \in I})$, where $\cE_i$ is a sheaf of $\cO_{U_i}$-modules and $\varphi_{ij}: (p^{ij}_j)^{\ast} \cE_j \to (p^{ij}_i)^{\ast} \cE_i$
are isomorphisms such that:
\begin{itemize}
    \item $\varphi_{ii} = \Id_{(p^{ii}_i)^{\ast} \cE_i}$ for all $i \in I$;
    \item $\varphi_{ij} = \varphi_{ji}^{-1}$ for all $i, j \in I$;
    \item $(p^{ijk}_{ij})^{\ast}(\varphi_{ij}) \circ (p^{ijk}_{jk})^{\ast}(\varphi_{jk}) \circ (p^{ijk}_{ki})^{\ast}(\varphi_{ki}) = \alpha_{ijk} \cdot \Id_{(p^{ijk}_{i})^{\ast} \cE_i} $ for all $i, j, k \in I$.
\end{itemize}

A \textit{morphism} $f: (\cE, \varphi) \to (\cE', \varphi')$ \textit{of $\alpha$-twisted sheaves} is the data of morphisms $(f_{i})_{i \in I}: \cE_i \to \cE'_i$ of $\cO_{U_i}$-modules such that the diagram

\[
\begin{tikzcd}
(p^{ij}_j)^{\ast} \cE_j \arrow{r}{\varphi_{ij}} \arrow{d}{(p^{ij}_j)^{\ast}(f_j)}
& (p^{ij}_i)^{\ast} \cE_i \arrow{d}{(p^{ij}_i)^{\ast}(f_i)} \\
(p^{ij}_j)^{\ast} \cE'_j \arrow{r}{\varphi'_{ij}}
& (p^{ij}_i)^{\ast} \cE'_i
\end{tikzcd}
\]
commutes for all $i, j \in I$.

\end{definition}

Clearly, the class of all $\alpha$-sheaves (given along the cover $\mathfrak{U}$) and morphisms defines a category that we denote by $\mathfrak{Mod}(X, \alpha, \mathfrak{U})$. It can be shown that if $\mathfrak{U'}$ is a refinement of $\mathfrak{U}$ on which $\alpha$ can be represented, then the categories $\mathfrak{Mod}(X, \alpha, \mathfrak{U})$ and $\mathfrak{Mod}(X, \alpha, \mathfrak{U'})$ are equivalent; we denote any of these categories by $\mathfrak{Mod}(X, \alpha)$. Moreover, if $\alpha$ and $\alpha'$ are in the same cohomology class, then the functor $(\cE, \varphi) \to (\cE, (\alpha^{-1} \alpha')\varphi)$ defines an equivalence of categories $\mathfrak{Mod}(X, \alpha)$ and  $\mathfrak{Mod}(X, \alpha')$, so we will use the notation $\mathfrak{Mod}(X, \alpha)$ for $\alpha \in \check{H}^2(X_{\et}, \Gm)$.

An $\alpha$-twisted sheaf $\cE$ is called quasi-coherent (coherent, locally free) if all the underlying sheaves are quasi-coherent (coherent, locally free, respectively). The corresponding category is denoted by $ \mathfrak{Qcoh}(X, \alpha)$ ($ \mathfrak{Coh}(X, \alpha)$, $\mathfrak{Loc}(X, \alpha)$, respectively).

\begin{definition}

An $X$-scheme $P$ is called a \emph{Severi--Brauer scheme} over $X$ if there exists a covering $\mathfrak{U} = (U_i \to X)_{i \in I}$ for the \'etale topology such that $P \times_X U_i \cong \PP^{n}_X \times_X U_i$ for some $n$ for all $i \in I$.

\end{definition}

If $\cE$ is an $\alpha$-twisted vector bundle (i. e. locally free $\alpha$-sheaf of finite rank) given along the cover $\mathfrak{U} = (U_i \to X)$, one can consider the schemes $\PP{\cE_i}$ over $U_i$ and isomorphisms $\textbf{P}(\varphi_{ij}): \PP{\cE_j} \times_{U_j} U_{ij} \to \PP{\cE_i} \times_{U_i} U_{ij}$. From the exactness of the sequence
\[
1 \to \Gm \to \GL_n \to \PGL_n \to 1
\]
we see that $\textbf{P}(\varphi_{ij}) \circ \textbf{P}(\varphi_{jk}) \circ \textbf{P}(\varphi_{ki})$ is the identity on $\PP{\cE_i} \times_{U_i} U_{ijk}$, which gives a descent datum on $X \times U$ (where $U = \bigsqcup U_i$). Using the ampleness of the anticanonical bundle on projective space, we conclude from \cite[VIII.7.8]{SGA1} that this data is effective. Thus, $\PP{\cE}$ is a Severi-Brauer scheme over $X$ with Brauer class $\alpha$. We have a bijection of sets:
\begin{gather*}
\{\text{isomorphism classes of Severi-Brauer schemes over } X \text{ of dimension } n  \} \leftrightarrow \\
\{\text{elements in } \check{H}^1(X_{\et}, \PGL_{n+1})  \} \leftrightarrow  \{\text{isomorphism classes of schemes }\\
 \PP{\cE} \text{ over } X, \text{ where } \cE \text{ is a twisted vector bundle of rank } n+1  \}.
\end{gather*}

\begin{proposition} \label{properties}

Let $\alpha, \alpha' \in \check{H}^2(X_{\et}, \Gm)$, and let $(\cE, \varphi)$ and $(\cF, \psi)$ be sheaves on $X$, twisted by $\alpha$ and $\alpha'$, respectively.

\begin{enumerate}
    \item The categories $\mathfrak{Mod}(X, \alpha)$ and $\mathfrak{Qcoh}(X, \alpha)$ are Abelian categories with enough injectives.
    \item  The categories $\mathfrak{Mod}(X, e)$ and $\mathfrak{Qcoh}(X, e)$ are the usual categories of $\cO_X$-modules and quasi-coherent sheaves on $X$, respectively.
    \item The sheaf $\cE \otimes \cF$, given by the gluing of $\cE_i \otimes \cF_i$ along $\varphi_i \otimes \psi_i$, is a well-defined $\alpha \alpha'$-sheaf on $X$. 
    \item Similarly, $\Sym^n(\cE)$ and $\bigwedge^n(\cE)$ are $\alpha^n$-sheaves on $X$.
    \item The sheaf $\underline{\Hom}(\cE, \cF)$, given by the gluing of $\underline{\Hom}(\cE_i, \cF_i)$ along $(\varphi_i^{-1})^{\vee} \circ \psi_i$, is a well-defined $\alpha^{-1} \alpha'$-sheaf on $X$. In particular, $\underline{\End}(\cE)$ is a regular (not twisted) sheaf on $X$.
    \item If $f: Y \to X$ is a morphism of schemes, then the sheaf $f^{\ast} \cE$, given by the data $(\{f^{\ast}\cE_i \}, \{f^{\ast}\varphi_{ij} \})$, is a well-defined $f^{\ast} \alpha$-sheaf on $Y$.
    
\end{enumerate}

\end{proposition}

\begin{proof}

See \cite[I.1-2]{caldararu} for details.

\end{proof}

In fact, the category of twisted sheaves is equivalent to the category of sheaf modules over an Azumaya algebra.

\begin{theorem}[{\cite[Theorems 1.3.5-1.3.7]{caldararu}}]

For any $\alpha$-twisted vector bundle $\cE$ over X, $\underline{\End}(\cE)$ is an Azumaya algebra with Brauer class $\alpha$. Conversely, if $\cA$ is an Azumaya algebra over $X$ with Brauer class $\alpha = [\cA ] \in \Br(X)$, then there exists an  $\alpha$-twisted vector bundle $\cE$ such that $\cA \cong \underline{\End}(\cE)$. For every such $\cE$ the functor
\[
-\otimes_{\cO_X} \cE^{\vee}: \, \mathfrak{Mod}(X, \alpha) \to \mathfrak{Mod}\text{-}\cA
\]
defines an equivalence between the category $\mathfrak{Mod}(X, \alpha)$ of $\alpha$-twisted sheaves and the category $\mathfrak{Mod}\text{-}\cA$ of sheaves of right $\cA$-modules on $X$.

\end{theorem}

\begin{corollary}

There exists an $\alpha$-twisted vector bundle on $X$ if and only if $\alpha \in \Br(X)$.

\end{corollary}

Given an Azumaya algebra $\cA$, we shall denote the corresponding Severi-Brauer scheme by $X_{\cA}$.

Similarly as for $\PP{\cE}$, one can define the Grassmannian $\Gr^{e}(\cE)$ of a twisted vector bundle $\cE$, which is the same as the $e$-th Severi-Brauer scheme attached to the Azumaya algebra $\underline{\End}(\cE)$.

\begin{proposition}[{\cite[Proposition 16]{Litt}}] \label{linebundle}

There exists an $\alpha$-twisted line bundle on $X$ if and only if $\alpha  = e \in \Br(X)$.

\end{proposition}




\begin{corollary} \label{rank}

Let $\cE$ be an $\alpha$-twisted vector bundle of rank $n$. Then $\alpha^n = e$ in $\Br(X)$.

\end{corollary}
\begin{proof}

By proposition \ref{properties}, $\bigwedge^n{\cE}$ is an $\alpha^n$-twisted line bundle, hence $\alpha^n = e$ by proposition \ref{linebundle}.

\end{proof}

\begin{corollary} \label{section}

Let $\pi: B \to S$ be a Severi-Brauer scheme over $S$. If $\pi$ admits a section, then $B = \PP{\cE}$ for some vector bundle $\cE$ over $S$.

\end{corollary}

\begin{proof}

Write $B =\PP{\cE}$ where $\cE$ is an $\alpha$-twisted bundle on $S$. The section to $\pi$ defines an $\alpha$-twisted line sub-bundle of $\cE$. Hence $\alpha = e$ by proposition \ref{linebundle}.

\end{proof}

The following proposition gives a description of the structure of Azumaya algebras over the spectrum of a field in terms of twisted sheaves.

\begin{proposition}[{\cite[Corollaries 19-21]{Litt}}] \label{fields}

Let $X = \Spec{K}$, where $K$ is a field.
\begin{enumerate}
    \item Let $\cE$ be an $\alpha$-twisted vector bundle over $X$. Then $\cE$ is simple if and only if $\End{\cE}$ is a division algebra over $K$.
    \item For any $\alpha \in \Br(X)$ there exists a unique isomorphism class $\cE_{\text{min}}$ of non-zero simple $\alpha$-twisted vector bundles over $X$.
    \item For any $\alpha$-twisted vector bundle $\cE$ over $X$ there exists $r \in \ZZ_{\ge 0}$ such that $\cE \cong \cE_{\text{min}}^{\oplus r}$. 
\end{enumerate}

\end{proposition}

For a Severi-Brauer variety $B = \PP{\cE}$ over $\Spec{K}$, the number $\rk(\cE_{\text{min}})$ is called the \textit{index} of $\PP{\cE}$, and will be denoted by $i(B)$. It is equal to the greatest common divisor of degrees of all $0$-cycles on $B$. The Severi-Brauer variety $B_{
\text{min}} \defeq \PP(\cE_{\text{min}})$ is the \textit{minimal twisted linear subvariety} of $B$.


For every Severi-Brauer scheme $B= \PP{\cE}$ we can associate its dual $B^{\text{op}} \defeq \PP(\cE^{\vee})$. If $B = X_{\cA}$, then $B^{\text{op}} = X_{\cA^{\text{op}}}$, where $\cA^{\text{op}}$ is an Azumaya algebra opposite to $\cA$.

\begin{theorem}[{\cite[Theorem 24]{Litt}}] \label{motivicdecomposition} 

Suppose $S$ is a variety over $\Spec{K}$ and 
\[
0 \to \cE_1 \to \cE_2 \to \cE_3 \to 0
\]
is a short exact sequence of $\alpha$-twisted vector bundles over $S$. Let $\rk(\cE_i) = r_i$. Then
\[
[\PP(\cE_2)] = [\PP(\cE_1)] + \LL^{r_1} [\PP(\cE_3)] = [\PP(\cE_3)] + \LL^{r_3} [\PP(\cE_1)] 
\]
in $K_0(\text{Var}/k)$.

\end{theorem}

\begin{corollary} \label{motive_of_a_sum}

Let $\cE$ be an $\alpha$-twisted vector bundle on $S$ such that $\cE = \cF^{\oplus r}$ for some $\alpha$-twisted vector bundle $\cF$ of rank $n$. Then
\[
[\PP(\cE)] = [\PP(\cF)](1+ \LL^n+ \cdots + \LL^{(r-1)n}). 
\]

\end{corollary}

\section{Automorphisms of symmetric powers}\label{section:automorphisms}

\subsection{Automorphisms of symmetric powers.}

Let $X$ be a quasi-projective variety over $K$.  Every automorphism $\varphi$ of $X$ yields an automorphism $\rho(f)$ of $\Sym^d(X)$ given by $(s_1, \ldots, s_d) 
\mapsto (f(s_1), \ldots, f(s_d))$. It is clear that the map $\rho: \Aut(X) \to \Aut(\Sym^d(X)), \; \varphi \mapsto \rho(\varphi)$ is a homomorphism of groups. In this section we prove the following:

\begin{theorem} \label{theorem:automorphisms}

Let $X$ be a smooth projective variety of dimension greater than $1$ over a field $K$ of characteristic zero. Assume that $\Pic(X_{\bar{K}}) = \ZZ$ and that $X_{\bar{K}}$ is simply connected. Then the natural map
\[
\rho: \Aut(X) \to \Aut(\Sym^d(X))
\]
is an isomorphism of groups.

\end{theorem}


\begin{lemma} \label{lemma:autofprod}

Let $X$ be a smooth irreducible projective variety over a field $K$ such that $\Pic(X) = \ZZ$ and $\Pic(X^d)$ is generated by the pullbacks of line bundles on $X$ for every $d > 0$. Then
\[
\Aut(X^d) \cong S_d \ltimes (\Aut(X))^{d}.
\]

\end{lemma}

\begin{proof}

Let $\cL$ be a very ample line bundle on $X$ and let $j: X \to \PP(\mathrm{H}^{0}(X, \cL)^{\vee})$ be a closed immersion such that $j^{\ast}(\cO(1)) \cong \cL$. We denote by $\pi_i: X^d \to X$ the projection on the $i$-th factor.

The group $S_d \ltimes (\Aut(X))^{d}$ acts on $X^d$
by the rule $(x_1, 
\ldots, x_d) 
\mapsto (f_{\sigma^{-1}(1)}(x_{\sigma^{-1}(1)}),\ldots, $ $f_{\sigma^{-1}(d)}(x_{\sigma^{-1}(d)}))$, where $\delta \in S_d, \, f_i \in \Aut(X)$. It is easily seen that this action is faithful, so we get an inclusion $S_d \ltimes (\Aut(X))^{d} \hookrightarrow \Aut(X^d)$. 

Conversely, let $\varphi$ be an automorphism of $X^d$. It induces an automorphism of $\Pic(X^d)$, so  there are isomorphisms
\[
\varphi^{\ast}(\pi_i^{\ast}(\cL)) \cong \bigotimes_{j=1}^d \pi_j^{\ast}(\cL)^{d_{ij}}
\]
for some $d_{ij} \in \ZZ$. Since $\varphi$ is an automorphism and $\Gamma(X, \cO_X) = K$, it follows that $h^{0}(X^d, \varphi^{\ast}(\pi_i^{\ast}(\cL))) = h^{0}(X^d, \pi_i^{\ast}(\cL)) = h^{0}(X, \cL)$. Note that $h^0(X, \cL^{k}) > h^0(X, \cL)$ for $k > 1$ and $h^0(X, \cL^{k}) = 0$  for $k < 0$. The Kunneth formula implies that $h^0(X^d, \bigotimes_{j=1}^d \pi_j^{\ast}(\cL)^{d_{ij}}) = \prod_{j=1}^d h^0(X, \cL^{d_{ij}})$, so the equality $h^0(X^d, \bigotimes_{j=1}^d \pi_j^{\ast}(\cL)^{d_{ij}}) =  h^{0}(X, \cL)$ holds if and only if there exist $\sigma(i) \in \{1, \ldots, d \}$ such that $d_{ij} = 1$ for $j = \sigma(i)$ and $d_{ij} = 0$ otherwise. Hence there are isomorphisms $\varphi^{\ast}(\pi_i^{\ast}(\cL)) \cong \pi_{\sigma(i)}^{\ast}(\cL)$ for $\sigma \in S_d$.

Consider the automorphism $\widetilde{\sigma}: X^d \to X^d$, $(x_1, \ldots, x_d) \mapsto (\sigma^{-1}(x_1), \ldots, \sigma^{-1}(x_d))$, and put $\psi = \widetilde{\sigma}^{-1} \circ \varphi$. Then $\psi^{\ast}(\pi_i^{\ast}(\cL)) \cong \pi_i^{\ast}(\cL)$ for every $i$ and $\varphi = \widetilde{\sigma} \circ \psi$. Let $N = h^{0}(X, \cL)-1$, $\PP^N = \PP(\mathrm{H}^{0}(X, \cL)^{\vee})$, and let $(s_0, \ldots, s_N) = (j^{\ast}(x_0), 
\ldots, j^{\ast}(x_N))$ be a basis of $\mathrm{H}^{0}(X, \cL)$, where $x_0, 
\ldots, x_N$ are hyperplane sections of $\PP^N$. The pullbacks $(\pi_i^{\ast}(s_k))_{k=0}^{N}$ form a basis of $\mathrm{H}^{0}(X^d, \pi_{i}^{\ast}(\cL))$ and, since $\psi^{\ast}(\pi_i^{\ast}(\cL)) \cong \pi_i^{\ast}(\cL)$, the $t_{ik} = \psi^{\ast}(\pi_i^{\ast}(s_k))_{k=0}^{N}$ must be another basis of $\mathrm{H}^{0}(X^d, \pi_{i}^{\ast}(\cL))$, hence we can write $t_{ik} = \sum_{j = 0}^{N} a_{kj}^{i}\pi_i^{\ast}(s_j)$, where $||a_{kj}^{i}||_{0 \le j, k \le N} \in \GL_N(K)$. We denote by $\widetilde{\psi}_i$ an automorphisms of $\PP^N$ given by the matrix $||a_{kj}^{i}||$ as an element of $\PGL_N(K)$. Now consider the following diagram:
\[
\xymatrix@R=15pt@C=15pt{  & & X^d  \ar[rr]^{\psi}  \ar[dd]^{\pi_i} & &  X^d \ar@{^{(}->}[dr]^{j^{d}} \ar[dd]^{\pi_i}  & \\
& & &  & & 
(\PP^N)^d \ar[dd]^{\pi_i} \\
 & &  X \ar@{^{(}->}[dr]^{j} \ar@{-->}[rr]^{\psi_i} & & X  \ar@{^{(}->}[dr]^{j} &  \\
& & & \PP^N  \ar[rr]^{\widetilde{\psi}_i} & & \PP^N }
\]
One has $j \circ \pi_i \circ \psi = \widetilde{\psi}_i \circ j \circ \pi_i$, since both maps are given by the global sections $(t_{ik})_{k=0}^N$ of $\mathrm{H}^{0}(X^d, \pi_{i}^{\ast}(\cL))$. Hence there exists a unique automorphism $\psi_i$ of $X$ making the diagram commutative. Thus $\psi$ acts on each factor of $X^d$, so $\varphi = \widetilde{\sigma} \circ \psi$ defines an element of $S_d \ltimes (\Aut(X))^{d}$. The lemma follows.

\end{proof}

\begin{proof}[Proof of the theorem \ref{theorem:automorphisms}] Let $n = \dim{X}$. Since $n > 1$, it follows that for $d > 1$ the small diagonal $X_d \subset \Sym^{d}(X)$ is the most singular locus in the sense that for closed points $x \in X_d$ and $y \in \Sym^{d}(X) 
\, \backslash \, X_d$ one has $\dim_{\kk(x)} T_{x}(\Sym^d(X)) = C_{n+d}^d - 1 > \dim_{\kk(y)} T_{y}(\Sym^d(X))$ (in order to prove that one may use theorem \ref{sym_poly} and the fact that $\hat{\cO}_x \cong K[[X_1, 
\ldots, X_d]]^{S_d}$, where $X_i = (x_{i,1}, \ldots, x_{i, n})$ and $S_d$ act by permuting the $X_i$). In particular, every automorphism of $\Sym^{d}(X)$ preserves the small diagonal, hence we get a map $\eta: \Aut(\Sym^d(X)) \to \Aut(X)$.

To show that $\rho$ and $\eta$ are inverse to each other, it is enough to work with $X_{\bar{K}}$. Note that $X_{\bar{K}}$ is irreducible: if $X_1, \ldots, X_n$ are irreducible components of $X_{\bar{K}}$, then there is a surjection $\Pic(X_{\bar{K}}) \to \bigoplus_{i=1}^n \Pic(X_i)$ and the condition $\Pic(X_{\bar{K}}) = \ZZ$ implies that $n=1$.

From now on, we write $X$ instead of $X_{\bar{K}}$. Let $\pi: X^d \to \Sym^{d}(X)$ be the quotient map. Note that $S_d$ acts without fixed points on $\widetilde{X}_{1, \ldots, 1}$, so the restriction $\left.\pi \right|_{\widetilde{X}_{1, \ldots, 1}}: \widetilde{X}_{1, \ldots, 1} \to X_{1, \ldots, 1}$ is \'etale. Since $X_{\bar{K}}$ is simply connected and $\codim(X^d \, \backslash \, \widetilde{X}_{1, \ldots, 1}) > 1$, it follows that $\widetilde{X}_{1, \ldots, 1}$ is simply connected (\cite[X.3.3]{SGA1}). Thus $\widetilde{X}_{1, \ldots, 1}$ is a universal covering space of $X_{1, \ldots, 1}$.

Every automorphism $\varphi \in \Aut(\Sym^d(X))$ preserve the smooth locus $X_{1, \ldots, 1}$, hence it lifts to an automorphism of the universal cover $\widetilde{X}_{1, \ldots, 1}$. Since $\pi$ is the normalization of $\Sym^{d}(X)$ in $K(\widetilde{X}_{1, \ldots, 1})$, we see that $\varphi$ extends to an automorphism $\widetilde{\varphi}$ of $X^d$. As $X_{\bar{K}}$ is smooth and simply connected, its first Betti number equals zero, so $\mathrm{H}^{1}(X, \cO_X) = 0$ and $\Pic(X^d) = \Pic(X)^d = \ZZ^d$ (\cite[Exercise~III.12.6]{Ha}). By lemma \ref{lemma:autofprod}, $\Aut(X^d) \cong S_d \ltimes (\Aut(X))^{d}$, and the only  automorphisms to descend to $\Sym^d(X)$ are those lying in the subgroup $S_d \times \Aut(X)$, where the second factor is the image of $\Aut(X)$ in $(\Aut(X))^{d}$ under the diagonal embedding. But $S_d$ acts trivially on $\Sym^{d}(X)$, so $\varphi \in \Aut(X)$.

\end{proof}

\begin{remark}

If $X$ is as in the theorem, then one can compute $\Aut(\Sym^{d}(X^r))$ by the same argument as in the proof of the theorem. For instance,  $\Aut(\Sym^2(X \times X)) \cong S_2 \ltimes \Aut(X \times X)$, where the nontrivial automorphism corresponds to the permutation of the first and the third factor. In particular, it shows that the theorem is no longer true if $\Pic(X) \neq \ZZ$. 

Similarly, it is possible compute $\Aut(\Sym^{d}(X \times Y))$, where $X$ and $Y$ are projective varieties of different dimensions such that $\Pic(X) = \Pic(Y) = \ZZ$ and one of them is simply connected, but the computations are quite elaborate. 

\end{remark}

\begin{example}

Let $Y$ be a quadric in $\PP_{\CC}^3$ (so one has $Y \cong \PP_{\CC}^1 \times \PP_{\CC}^1$), and let $q$ be the quadratic form associated with $Y$. The nontrivial automorphism of $\Sym^2(Y)$ has a nice geometric interpretation. Note that taking orthogonal relative to the scalar product defined by $q$ yields an automorphism of the Grassmannian $\Gr_1(\PP^3)$. It preserves all the tangent spaces to $Y$ and hence defines a nontrivial automorphism of $\Sym^2(Y)$ such that its restriction to the diagonal is trivial.

\end{example}





\section{Some classes of Severi-Brauer varieties in \texorpdfstring{$K_0(\text{Var}/K)$}{TEXT}}

\subsection{General observations.}

We compute below some relations between the classes of Severi-Brauer varieties in $K_0(\text{Var}/K)$.

\begin{proposition} \label{Neo}

Let $B$ be a Severi-Brauer variety of dimension $m-1$ over a field $K$. Let $n = \dim(B_{\text{min}}) + 1$ and $r = m/n$. Then
\[
[B] = [B_{\text{min}}](1+\LL^{n}+\LL^{2n}+ \ldots + \LL^{(r-1)n}).
\]

\end{proposition}

\begin{proof}

We have $B_{\text{min}} = \PP{\cE_{\text{min}}}$ for some simple twisted vector bundle $\cE_{\text{min}}$ on $\Spec{K}$ of rank $n$, and $B = \PP{(\cE_{\text{min}}^{\oplus r})}$. Thus the proposition follows from corollary \ref{motive_of_a_sum}.


\end{proof}

\begin{proposition} \label{class_of_product}

Let $B$ be a Severi-Brauer variety over $\Spec{K}$ of dimension $n$. Then for every $d \ge 0$
\begin{gather*}
[B^d] = [B] \cdot [\PP^{n}]^{d-1}, \\
[B \times B^{\text{op}}] = [B] \cdot [\PP^{n}] = [B^{\text{op}}] \cdot [\PP^{n}]
\end{gather*}
in $K_0(\text{Var}/K)$.

\end{proposition}

\begin{proof}

Let $B = \PP{\cE}$ for some $\alpha$-twisted bundle $\cE$ over $\Spec{K}$, and let $\pi: B \to \Spec{K}$ be the structure map. Observe that $B \times B$ is a Severi-Brauer scheme over $B$. The map $B \times B \to B$ admits a section (the diagonal map), so, by corollary \ref{section}, $B \times B$ is a trivial Severi-Brauer variety over $B$. Since $B \times B \cong \PP_B(\pi^{\ast}\cE)$, it follows that $\pi^{\ast} \alpha = e$ in $\Br(B)$. In particular, $[B \times B] = [B] \cdot [\PP^n]$, and we obtain the first formula by induction.

To prove the second equality, note that $B \times B^{\text{op}} =  \PP{\cE} \times  \PP{\cE^{\vee}} = \PP_B(\pi^{\ast}\cE^{\vee}) = \PP_B((\pi^{\ast}\cE)^{\vee})$. The sheaf $\pi^{\ast}\cE^{\vee}$ is twisted by $(\pi^{\ast} \alpha)^{-1} = e$, hence $B \times B^{\text{op}} $ is trivial over $B$ as well. Thus $[B \times B^{\text{op}}] = [B] \cdot [\PP^{n}]$.

\end{proof}

\begin{observation}

Let $B$ be a non-trivial Severi-Brauer variety over $\Spec{K}$ of dimension $n$ and assume that $\Char{K} = 0$. Then $[B] \neq [\PP^{n}]$ in $K_0(\text{Var}/K)$.

\end{observation}

\begin{proof}

Assume that $[B] = [\PP^{n}]$. Then $B$ and $\PP^{n}$ are stably birational by theorem \ref{Lar_Lun}. Hence $B \times \PP^m$ is birational to $\PP^{m+n}$ for some $m \ge 0$. In particular $(B \times \PP^m)(K) \neq \varnothing$, so $B(K) \neq  \varnothing$, which is a contradiction with the choice of $B$.

\end{proof}

\begin{remark}

By proposition \ref{class_of_product}, $[B] \cdot ([B] - [\PP^n]) = 0$, so $[B]$ is a zero divisor in $K_0(\text{Var}/K)$ for any nontrivial Severi-Brauer variety $B$.

\end{remark}

\subsection{The classes of Severi-Brauer surfaces}

\begin{proposition} \label{surfaces}

Let $B$ be a non-trivial Severi-Brauer surface over a field $K$ of characteristic not equal to 2 or 3, and let $B^{\text{op}}$ the dual surface. Then $[B] = [B^{\text{op}}]$ in $K_0(\text{Var}/K)$.

\end{proposition}

\begin{proof}

Let $P \in B$ be a point of degree 3. Then the three points of the set $P_{k^{sep}} \subset \PP_{k^{sep}}^2$ are in general position, i. e. they are not contained in a line (otherwise the line would be $\Gal(k^{sep}/k)$-invariant, and hence it would be defined over $K$, which is impossible (see, for example, \cite[Lemma 2.6]{Shramov})).

Let $\pi: X \to B$ be the blow-up of $B$ with the center at $P$, and let $E$ be its exceptional divisor. Denote by $E'$ the union of three lines on $B_{k^{sep}} \cong \PP^2_{k^{sep}}$ passing through the pairs of points of point of $P_{k^{sep}}$. Since $G = \Gal(k^{sep}/k)$ permutes the points in $P_{k^{sep}}$, it follows that $E'$ is defined over $K$. Let $\pi': X \to B'$ be a map obtained by blowing $E'$ down. It is not hard to show that $B' \cong B^{\text{op}}$, see \cite[Lemma 1]{Wie}. Thus we obtain a birational map $\eta: B \to B'$.

We have $E = \PP^1_{k} \times_{\Spec{K}} \Spec{\kk(P)}$ and $E' = \PP^1_{k} \times_{\Spec{K}} \Spec{F'}$, where $ \kk(P), F'$ are separable field extensions of $K$ of degree $3$. Moreover, $F' = (\kk(P) \otimes_{k} k^{sep})^{G} = \kk(P)$, hence $E \cong E'$. This can be seen explicitly as well using that, after extending scalars, the map $\eta$ is the standard Cremona transformation
\[
\eta_{k^{sep}}: \PP^2_{k^{sep}} \dashrightarrow \PP^2_{k^{sep}}, \quad [x:y:z] \mapsto [yz: xz: xy]
\]
in a suitable basis.

Thus $[B] = [X] - [E] + [\kk(P)]= [X] - [E'] + [F'] = [B']$.

\end{proof}

\subsection{The class of the symmetric square}

Given two distinct points on the projective space, they span a line, which gives a natural map $\rho: \Sym^{2}(\PP^n) \, \backslash \, D  \to \Gr_1(\PP^n)$, where $D \cong \PP^n$ is the
diagonal. The fiber over each point $l \in \Gr_1(\PP^n)$ consists of pairs of point lying on $l$, and is isomorphic to $\Sym^{2}(\PP^1)  \, \backslash \, \PP^1 \cong \AA^2$. Moreover, $\rho$ has locally a section: fix two lines $l_1, l_2$ on $\PP^n$; then every other line intersects both $l_1$ and $l_2$ in one point, which gives a point on $\Sym^{2}(\PP^n)$. Using this observations, it is easy to prove that $\rho$ is a Zariski locally trivial fibration with fiber $\AA^2$.

We exploit the analogous map to find the class in $K_0(\text{Var}/K)$ of the symmetric square of a Severi-Brauer variety.

We use the following notation: if $X$ is an $Y$-scheme, and $Y$ is a $Z$-scheme, we write $\Sym^{n}_Y (X)$ for the symmetric product over $Y$ and  $\Sym^{n}_Z (X)$ for the symmetric product over $Z$.

\begin{proposition} \label{square}

Let $B$ be a Severi-Brauer variety of even dimension over $\Spec{K}$. Then
\[
[\Sym^{2}(B)] = [B] + [\Gr_1(B)]\LL^2.
\]

\end{proposition}

\begin{proof}

For every scheme $S$  over $K$ and  locally free $\cO_S$-module $\cF$ there is a natural map $\rho: \Sym^2(\PP(\cF))  \, \backslash \, \PP(\cF) \to  \Gr^2(\cF)$, given on $T$-valued points by 
\[
\begin{aligned}
\rho(T): (\PP{(\cF)} \times \PP{(\cF)})(T)  & \dashrightarrow \Gr^2{(\cF)}(T), \\
([u_1: h^{\ast}\cF \twoheadrightarrow \mathcal{L}_1], [u_2: h^{\ast}\cF \twoheadrightarrow \mathcal{L}_2])  & \mapsto (
(u_1 \oplus u_2) \circ h^{\ast}(i):
h^{\ast} \cF \hookrightarrow h^{\ast} \cF \oplus h^{\ast} \cF \twoheadrightarrow \mathcal{L}_1 \oplus \mathcal{L}_2),
\end{aligned}
\]
where $h: T \to S$ and $i: \cF \to \cF \oplus \cF$ is the diagonal map. Let $D \cong \PP{(\cF)}$ be the diagonal in $\PP{(\cF)} \times \PP{(\cF)}$. It is easily checked that $\rho$ is defined on $\PP{(\cF)} \times \PP{(\cF)} \, \backslash \,  D$ and invariant under the action of $S_2$. Since $\cF$ is locally free, it follows from the discussion above that $\rho$ is a Zariski locally trivial $\AA^2$-fibration.

Let $B \cong \PP{\cE}$, where $\cE$ is an $\alpha$-twisted bundle $\cE$ over $\Spec{K}$ of rank $2n+1$. The morphism $\rho$ constructed above yields a map $\rho': \Sym^2(\PP(\cE))  \, \backslash \, \PP(\cE) \to  \Gr^2(\cE)$, which is an $\AA^2$-fibration is the \'etale topology. We claim that it is locally trivial in the Zariski topology as well. 

Let $\pi:  \Gr^2{\cE} \to \Spec{K}$ be the structure map. The diagonal map $\Gr^2{\cE} \to \Gr^2{\cE} \times  \Gr^2{\cE}$ is a section to the projection $\Gr^2{(\pi^{\ast}\cE)} \to \Gr^2{(\cE)}$, so it defines an $(\pi^{\ast}\alpha)$-twisted sub-bundle of $\pi^{\ast}\cE$ of rank 2. In particular, one has $(\pi^{\ast}\alpha)^2 = e$ in $\Br(\Gr^2{\cE})$ by corollary \ref{rank}. On the other hand, $\alpha^{2n+1} = e$ by the same corollary, so $ \pi^{\ast}\alpha =(\pi^{\ast}\alpha)^{2n+1} = \pi^{\ast}(\alpha^{2n+1}) = e $. Therefore, by the split case, the map
\[
\Sym^2_{\Gr^2{\cE}}(\PP{\cE} \times \Gr^2{\cE})  \, \backslash \, (\PP{\cE} \times \Gr^2{\cE}) \to \PP{\cE}^{\vee} \times \Gr^2{\cE}
\]
is an $\AA^2$-bundle is the Zariski topology. Note that $\Sym^2_{\Gr^2{\cE}}(\PP{\cE} \times \Gr^2{\cE}) \cong \Sym^2_{K}(\PP{\cE}) \times \Gr^2{\cE}$, so, pulling back along the diagonal map $\Gr^2{\cE}\to \Gr^2{\cE} \times \Gr^2{\cE}$, we see that $\rho'$ is locally trivial in the Zariski topology.

Thus, we have $[\Sym^2(B)] - [B] = \LL^2 \cdot [\Gr_1(B)]$, and the assertion of the proposition follows.

\end{proof}

\begin{corollary} \label{sq_sur}

If $B$ is a Severi-Brauer surface over a field of characteristic not equal to 2 or 3, then $[\Sym^2(B)] = [B](1+ \LL^2) = [B^2]-\LL \cdot [B]$.

\end{corollary}

\begin{proof}

This is a direct consequence of propositions \ref{surfaces}, \ref{square} and \ref{class_of_product}.

\end{proof}

\subsection{The classes of some higher symmetric powers}

\begin{lemma} \label{split}

Let $K$ be a field and let $\alpha \in \Br(K)$. Assume that $\pi: S \to \Spec{K}$ is such that $\pi^{\ast}\alpha = e$ in $\Br(S)$. Then for every Severi-Brauer variety $B$ of dimension $n$ over $K$ with Brauer class $\alpha$ one has
\[
[\Sym^{d}(B)]\cdot [S] = [\Sym^{d}(\PP^n)] \cdot [S]
\]
for every $d \ge 0$.

\end{lemma}

\begin{proof}

Let $B = \PP{\cE}$. As $\pi^{\ast}\alpha = e$, the sheaf $\pi^{\ast}\cE$ is regular (not twisted), hence $[S \times B] = [\PP(\pi^{\ast}\cE)] = [S \times \PP^n]$. Considering $S \times B$ as a variety over $S$, we see that $[B_S]_S = [\PP_S^n]_S$ in $K_0(\text{Var}/S)$, thus $[\Sym^d(B_S)]_S = [\Sym^d(\PP_S^n)]_S$. Since $\Sym^d(B_S) \cong \Sym^d(B) \times S$ and $\Sym^d(\PP_S^n) \cong \Sym^d(\PP^n) \times S$, it follows that $[\Sym^{d}(B)] \cdot [S] =  [\Sym^{d}(\PP^n)] \cdot [S]$ in $K_0(\text{Var}/K)$.

\end{proof}

\begin{proposition} \label{highsym}

Let $B$ be a Severi-Brauer variety over $K$ of dimension $n$ and let $d$ be the order of $\alpha = [B]$ in $\Br(K)$. Then
\[
[\Sym^{r}(B)] \cdot[\PP^{n}] = [\Sym^{r}(\PP^{n})] \cdot [B]
\]
for every $r$ coprime with d.

\end{proposition}

\begin{proof}

Let $B = \PP{\cE}$ and let $\pi_r: \Sym^{r}(B) \to \Spec{K}$ for $r \ge 0$. We have $[B \times \Sym^{r}(B)] = [B] \cdot [\Sym^{r}(B)] = [B] \cdot [\Sym^{r}(\PP^{n})]$ by lemma \ref{split}.
On the other hand, $B \times \Sym^{r}(B) \cong \PP_{\Sym^{r}(B)}(\pi_r^{*}\cE)$. Consider the diagram
\[\begin{tikzcd}
\Sym^{r}(B) \arrow[hookrightarrow]{r}{\varphi}  \arrow{rd}{\pi_r} 
  & \PP(\Sym^{r}\cE) \arrow{d}{h} \\
    & \Spec{K}
\end{tikzcd}
\]
where $\varphi$ is the map that will be constructed in the course of the proof of proposition \ref{emb}. By proposition \ref{properties}, $\Sym^{r}\cE$ is twisted by $\alpha^{r}$, so the Brauer class of $\PP(\Sym^{r}\cE)$ is $\alpha^r$. In particular, $h^{\ast} (\alpha^r) = e$ in $\Br(\PP(\Sym^{r}\cE))$. Let $s, t$ be integers such that $sr+td = 1$. Then $\alpha = (\alpha^{r})^s$, so $h^{\ast} \alpha = (h^{\ast} (\alpha^r))^s = e$ and $\pi_r^{\ast}\alpha = \varphi^{\ast}(h^{\ast}\alpha) = e$. It follows that $[\PP_{\Sym^{r}(B)}(\pi_r^{\ast}\cE)] = [\Sym^{r}(B)] \cdot [\PP^n]$.

\end{proof}

\subsection{Symmetric powers of projective space.} It is interesting that the classes of $\Sym^{d}(\PP_{k}^n)$ in $K_0(\text{Var}/K)$ coincide with the classes of Grassmannians. Since both classes are well-known, the equality can be derived formally, but we sketch a short geometric proof.


Let $\Gr_{d, n} = \Gr^d((\cO_{\Spec{K}}^n)^{\vee})$ be the Grassmannian over $K$. Recall that the Gaussian binomial coefficient $\binom{n}{d}_q$ is a polynomial in $q$ with integer coefficients which is defined by the formula:
\[
\binom{n}{d}_q = \frac{(q^n-1)(q^n-q) \cdots (q^n-q^{d-1})}{(q^d-1)(q^d-q) \cdots (q^d-q^{d-1})}.
\]

Clearly, if $k=\FF_q$ is a finite field, then $|\Gr_{d, n}(k)| = \binom{n}{d}_q$. We put $p_{d, n}(t) = \binom{n+d}{d}_t \in \ZZ[t]$.

\begin{proposition} \label{Grassmannian}

The classes of $\Sym^{d}(\PP^n)$ and $\Gr_{d, n+d}$ in $K_0(\text{Var}/k)$ are equal. More precisely:
\[
[\Sym^{d}(\PP^n)] = [\Gr_{d, n+d}] = p_{d, n}(\LL).
\]
\end{proposition}
\begin{hproof}

Let $V$ be an $n+1$ dimensional vector space; choose a one-dimensional subspace $V_1 \subset V$. The projection $V \to V/V_1$ induces a regular map $\Gr_{d, \, n+1} \, \backslash \, \Gr_{d-1, n} \to \Gr_{d, n}$. It is easy to see that this is a locally trivial $\AA^d$-fibration, and that the standard affine open cover is a trivializing cover. Therefore $[\Gr_{d, \, n+1}] = [\Gr_{d, n}]\LL^d + [\Gr_{d-1, n}]$.

Now consider the projection $\PP^{n} \, \backslash \, \{x\} \to \PP^{n-1}$ with center $x \in \PP^{n}$. It induces a morphism $\pi: \Sym^{d}(\PP^{n}) \, \backslash \, \Sym^{d-1}(\PP^{n}) \to \Sym^{d}(\PP^{n-1})$. It is easily seen that $\pi^{-1}(\PP^{n-1}_{d_1, \ldots, d_r})$ has a structure of Zariski $\AA^d$-fibration. When we consider all the tuples $(d_1, \ldots, d_r)$ with $d_1 \le d_2 \le \ldots \le d_r$ and $d_1+\ldots+d_r = d$, the $\PP^{n-1}_{d_1, \ldots, d_r}$ give a partition of $\Sym^{d}(\PP^{n-1})$, hence we have $[\Sym^{d}(\PP^{n})] = [\Sym^{d}(\PP^{n-1})] \LL^{d} + [\Sym^{d-1}(\PP^{n})]$. The assertion of the proposition now follows by double induction on $d$ and $n$.

\end{hproof}





\begin{remark}

Note that if $n > 1, d > 1$, then $\Sym^{d}(\PP^n)$ is not smooth. In particular, $\Sym^{d}(\PP^n)$ and $\Gr_{d, n+d}$ are not isomorphic for $n, d > 1$.

\end{remark}




\begin{conjecture}

Let $\cE$ be an $\alpha$-twisted locally free sheaf on $\Spec{K}$ of rank $n$. It is natural to expect that $[\Sym^{nk+1}(\PP{\cE})] = [\Gr^{nk+1}(\cE^{\oplus(k+1)})]$ for every $k \ge 0$.

\end{conjecture}





\section{Motivic Zeta functions}

\subsection{Rationality questions.} By $(n, m)$ we denote the greatest common divisor of two integers $n, m$.
 
\begin{proposition} \label{ratinoalinquotient}

Let $B$ be a Severi-Brauer variety of index $n = i(B)$ over a field $K$ of characteristic zero. Then the image of $Z_{\text{mot}}(B, t)$ in $K_0(\text{Var}/K)/(\LL)$ is a rational function.

\end{proposition}

\begin{proof}

For every positive integer $d$ there are stable birational equivalences $\Sym^{d+n}(X_B) \overset{\text{\text{stab}}}{\sim} \Sym^{(d+n, \, n)}(B) = \Sym^{(d, \, n)}(B) \overset{\text{\text{stab}}}{\sim} \Sym^{d}(B)$ (see, for instance, \cite[Theorem 1]{Kollar}). On the other hand, $K_0(\text{Var}/K)/(\LL) \cong \ZZ[\text{SB}/K]$ by theorem \ref{Lar_Lun}. Therefore, the classes of $\langle \Sym^{d}(B) \rangle$ and $\langle \Sym^{d+n}(B) \rangle$ in $K_0(\text{Var}/K)/(\LL)$ are equal, so
\[
Z_{\text{mot}}(B, t)_{\mmod{\LL}} = \frac{1}{1-t^n} \sum_{i=0}^{n-1} \langle \Sym^{i}(B) \rangle t^i = \frac{1}{1-t^n} \sum_{i=0}^{n-1} \langle \Sym^{(i, n)}(B) \rangle t^i.
\]
\end{proof}

\begin{observation} \label{rationality_of_matrix}

Let $B$ be a Severi-Brauer variety of dimension $m-1$ over a field $K$. Let $n = \dim(B_{\text{min}}) + 1$ and $r = m/n$. If $Z_{\text{mot}}(B_{\text{min}}, t)$ is rational, then $Z_{\text{mot}}(B, t)$ is rational.

\end{observation}

\begin{proof}

It follows from propositions \ref{zeta1}, \ref{zeta2} and \ref{Neo} that
\[
Z_{\text{mot}}(B, t) = \prod_{i=0}^{r-1}Z_{\text{mot}}(B_{\text{min}}, \, \LL^{in} t).
\]

\end{proof}

\begin{proposition}

If $X$ and $Y$ are birational smooth irreducible complete schemes of dimension $2$ over a field $K$ of characteristic zero, then $Z_{\text{mot}}(X, t)$ is rational if and only if $Z_{\text{mot}}(Y, t)$ is rational.

\end{proposition}

\begin{proof}

By \cite[Theorem 0.3.1]{AKMW}, every birational map between $X$ and $Y$ can be factored into a sequence of blowings up and blowings down with centers in nonsingular points. But if $\psi: U  \dashrightarrow V$ in a morphism of two surfaces obtained by blowing up a smooth point $x \in V$, then $[U] - [\PP^{1}_{k}] \cdot [\Spec{\kk(x)}]= [V] - [\Spec{\kk(x)}]$, hence
\[
Z_{\text{mot}}(U, t) = Z_{\text{mot}}(V, t) \cdot Z_{\text{mot}}(\Spec{\kk(x)}, \LL t).
\]
One can show that the motivic zeta function of a zero dimensional scheme over a field of characteristic zero is always rational, so $Z_{\text{mot}}(U, t)$ is rational if and only if $Z_{\text{mot}}(V, t)$ is rational.

\end{proof}

\subsection{Embedding of symmetric powers}

In this section we prove the existence of the embedding $\Sym^{d}(\PP{\cE}) \to \PP{(\Sym^{d}{\cE})}$ that was used in the course of the proof of proposition \ref{highsym}.

\begin{proposition} \label{emb}

Let $S$ be a scheme over $\Spec{K}$, where $K$ is a field of characteristic zero, a let $\cE$ be an $\alpha$-twisted locally free sheaf of finite rank on $S$ for $\alpha \in \Br(S)$. Then for every $d \in \ZZ_{\ge 0}$, there exists a closed immersion
\[
\varphi: \Sym^{d}(\PP{\cE}) \to \PP{(\Sym^{d}{\cE})}.
\]

\end{proposition}

Before proving the proposition, let us introduce some notation. 

Let $R$ be a ring. For any $R$-module $M$ the symmetric group $S_d$ acts on $M^{\otimes d}$ by
\[
\sigma(m_1 \otimes \cdots \otimes m_d) = m_{\sigma^{-1}(1)} \otimes \cdots \otimes m_{\sigma^{-1}(d)} \quad \forall \, \sigma \in S_k.
\]

We denote by $\mathcal{ST}^{d}(M)$ the $R$-submodule of symmetric $d$-tensors:
\[
\mathcal{ST}^{d}(M) = \{z \in M^{\otimes d} \mid \sigma z = z \; \, \forall \, \sigma \in S_k \}.
\]

It is a well-known fact from commutative algebra that if $d!$ is a unit in $R$, then the map
\begin{equation} \label{sym}
Sym: \Sym^d(M) \to \mathcal{ST}^{d}(M), \quad z \mapsto \sum_{\sigma \in S_d} \sigma z
\end{equation}
is an $R$-module isomorphism.

For technical convenience, we will also use the map
\begin{equation} \label{sym'}
Sym': \Sym^d(M) \to \mathcal{ST}^{d}(M), \quad z \mapsto \sum_{\substack{\sigma \in S_d \\ \sigma z \neq \tilde{\sigma} z}} \sigma z,
\end{equation}
where the sum is over all different tensors $\sigma z$ (for example, $Sym'(m_1 \otimes m_1 \otimes m_2) = Sym(m_1 \otimes m_1 \otimes m_2)/2$).

If $(X, \cO_X)$ is a scheme and $\cF$ is an $\cO_X$-module, we denote by $\mathcal{ST}^{d}(\cF)$ the sheaf attached to the presheaf $U \mapsto \mathcal{ST}^{d}(\Gamma(U, \cF))$. It is clear that if $X$ is a scheme over a field of characteristic zero, then the map $\Sym^{d}(\cF) \to \mathcal{ST}^{d}(\cF)$ induced by (\ref{sym}) or (\ref{sym'}) is an isomorphism of $\cO_X$-modules.

\begin{lemma} \label{surj}

Let $(R, \fm)$ be a local ring. Suppose we are given $d$ surjective morphisms $(\varphi_i: R^n \to R)_{i=1}^{d}$ of $R$-modules. Then the composition
\[
(\varphi_1 \otimes \cdots \otimes \varphi_d) \circ i: \; \, \mathcal{ST}^{d}(R^n) \overset{i}{\hookrightarrow} (R^n) \otimes_R \cdots \otimes_R (R^n) \to R \otimes_R \cdots \otimes_R R
\]
is surjective.

\end{lemma}

\begin{proof}

We put $\varphi = \varphi_1 \otimes \cdots \otimes \varphi_d$. Let us choose a basis $(e_1, \ldots, e_n)$ in $R^n$, and let $\varphi_i(e_j) = r_j^{i}$. Since $\varphi_i$ is surjective, it follows that $(r_1^{i}, \ldots, r_n^{i}) = R$.

Let $1 \le j_1 \le j_2 \le \ldots \le j_d \le n$ be natural numbers, and let $m_1 < \ldots < m_s$ be such that the first $l_1$ of the $j_i$ are equal to $m_1$, the next $l_2$ of the $j_i$ are equal to $m_2$, and so on. For $\sigma \in S_d$ and $r^{1}_{j_1}, \ldots, r^{d}_{j_d} \in R$ we put $\sigma(r^{1}_{j_1} \cdot \ldots \cdot r^{d}_{j_d}) = r^{1}_{j_{\sigma^{-1}(1)}} \cdot \ldots \cdot  r^{d}_{j_{\sigma^{-1}(d)}}$ and
\[
Sym'(r^{1}_{j_1} \cdot \ldots \cdot r^{d}_{j_d}) =  \frac{1}{l_1! \ldots l_s!} \sum_{\sigma \in S_d} \sigma(r^{1}_{j_1} \cdot \ldots \cdot r^{d}_{j_d}).
\]
It is clear that $Sym'$ is well-defined. Identifying $R^{\otimes d}$ with $R$ via $a_1 \otimes \ldots \otimes a_d \mapsto a_1 \cdot \ldots \cdot a_d$, we see that
\[
\varphi \bigl(Sym'(e_{j_1} \otimes \cdots \otimes e_{j_d})\bigr) = Sym'(r^1_{j_1} \cdot \ldots \cdot r^d_{j_d}).
\]
It is enough to prove that there exist $1 \le j_1 \le j_2 \le \ldots \le j_d \le n$ such that $Sym'(r^1_{j_1} \cdot \ldots \cdot r^d_{j_d}) \notin \fm$. Indeed, in this case $\varphi(Sym'(e_{j_1} \otimes \cdots \otimes e_{j_d}))$ is a unit in $R$ and the lemma follows.

Assume that $Sym'(r^1_{j_1} \cdot \ldots \cdot r^d_{j_d}) \in \fm$ for every $1 \le j_1 \le j_2 \le \ldots \le j_d \le n$. Consider the polynomial ring $R[t_1, \ldots, t_n]$. It is easily seen that
\[
f(t_1, \ldots, t_n) \defeq \prod_{i=1}^n (r_1^{i}t_1+\cdots+r_n^{i}t_n) = \sum_{1 \le j_1 \le \ldots \le j_d \le n} Sym'(r^1_{j_1} \cdot \ldots \cdot r^d_{j_d})t_{j_1} \ldots t_{j_d}. 
\]
Therefore, $f \in \fm R[t_1, \ldots, t_n]$. Note that the ideal $\fm R[t_1, \ldots, t_n]$ is prime, since one has $R[t_1, \ldots, t_n]/\fm R[t_1, \ldots, t_n] \cong (R/\fm R)[t_1, \ldots, t_n]$, hence $r_1^{i}t_1+\cdots+r_n^{i}t_n \in \fm R[t_1, \ldots, t_n]$ for some $i$. But then $r_{j}^{i} \in \fm$ for every $j \in \{1, \ldots, n \}$, which is a contradiction with the surjectivity of $\varphi_i$.

\end{proof}

\begin{proof}[Proof of proposition \ref{emb}]

Assume first that $\cE$ is a regular (not-twisted) sheaf on $S$. Let $Sym': \Sym^{d}(\cE) \xrightarrow{\sim} \mathcal{ST}^{d}(\cE)$ be an isomorphism described above, and let $i: \mathcal{ST}^{d}(\cE) \hookrightarrow \cE^{\otimes d}$ be an inclusion. Recall that for every $S$-scheme $h: T \to S$ we have $\Sym^d(h^{\ast}\cE) = h^{\ast}(\Sym^d(\cE))$, $\mathcal{ST}^{d}(h^{\ast}\cE) = h^{\ast}(\mathcal{ST}^{d}(\cE))$ and  $h^{\ast}(\cE^{\otimes d}) = (h^{\ast}\cE)^{\otimes d}$. We define
\[
\begin{aligned}
\tilde{\varphi}(T): \PP(\cE)^{d}(T) \to & \; \PP(\Sym^{d} \cE)(T), \\
(u_i: h^{\ast}\cE \twoheadrightarrow \mathcal{L}_i)_{i=1}^d  \mapsto & \; ((u_1 \otimes \cdots \otimes u_d) \circ h^{\ast}(i \circ Sym'): \\
& h^{\ast}( \Sym^{d}\cE) \xrightarrow{\sim} h^{\ast} (\mathcal{ST}^{d}(\cE)) \hookrightarrow h^{\ast}(\cE^{\otimes d}) \twoheadrightarrow \mathcal{L}_1 \otimes \cdots \otimes \mathcal{L}_d). 
\end{aligned}
\]
Since $\cE$ and $\mathcal{L}_1, \ldots, \mathcal{L}_d$ are locally free, it follows from lemma \ref{surj} that the composition $\mathcal{ST}^{d}(h^{\ast}\cE) \to (h^{\ast}\cE)^{\otimes d} \twoheadrightarrow \mathcal{L}_1 \otimes \cdots \otimes \mathcal{L}_d$ is surjective on stalks, so the map $\tilde{\varphi}(T)$ is well-defined. Clearly, this map is functorial in $T$ and invariant under the action of $S_d$ on $\PP(\cE)^{d}$. Hence we obtain a morphism $\varphi:  \Sym^{d}(\PP{\cE}) \to \PP{(\Sym^{d}{\cE})}$ of $S$-schemes.

If $(\cE, \psi)$ is an $\alpha$-twisted vector bundle, choose a trivializing \'etale cover $(U_i \to S)_{i \in I}$. Since quotients by finite groups commute with flat extensions (\cite[V.1.9]{SGA1}), we have $\Sym^{d}(\PP{\cE}) \times_S U_i \cong \Sym^{d} \PP(\cE_i)$. It is easy to check that the diagram
\[
\begin{tikzcd}
\Sym^{d}(\PP{\left. \cE_j \right|_{U_{ij}}})
\arrow{r}{} \arrow{d}{\Sym^{d}(\textbf{P}(\psi_{ij}))}
& \PP(\Sym^{d}(\left. \cE_j \right|_{U_{ij}})) \arrow{d}{\textbf{P}(\Sym^{d} \psi_{ij})} \\
\Sym^{d}(\PP{\left. \cE_i \right|_{U_{ij}}})
\arrow{r}{}
& \PP(\Sym^{d}(\left. \cE_i \right|_{U_{ij}}))
\end{tikzcd}
\]
is commutative, so we get a well-defined map $\varphi:  \Sym^{d}(\PP{\cE}) \to \PP{(\Sym^{d}{\cE})}$.

To show that $\varphi$ is a closed immersion, we may assume that $S = \Spec{R}$ and $\PP(\cE) = \PP_R^{n}$, where $R$ is an algebra over an algebraically closed field of characteristic zero (since the property of being a closed immersion is fpqc local on the target and stable under faithfully flat descent). In this case, the map is given on $A$-valued points by
\[
(A^{n+1})^{d} \ni (r^{1}_0, \ldots, r^{1}_n, r^{2}_0, \ldots, r^{2}_n, \ldots, r^{d}_n) \mapsto (Sym'(r^1_{j_1}, \ldots, r^{d}_{j_d}))_{0 \le j_1 \le \ldots \le j_d \le n} \in A^{N},
\]
where $Sym'$ as in the proof of the lemma, $N = C_{d+n}^d$. Since sets of the form $\Sym^{d}(\AA_R^{n})$ form an affine open cover of $\Sym^{d}(\PP_R^{n})$, it is enough to prove that the map
\[
R[t_{k_1, \ldots, k_n}]_{(1 \le \sum_{i=1}^n k_i \le d)} \to R[X_1, \ldots, X_d]^{S_d}, \quad t_{k_1, \ldots, k_n} \mapsto e_{k_1, \ldots, k_n}(X_1, \ldots, X_n),
\]
where $X_i = (x_{i, 1}, \ldots, x_{i, n})$, is surjective. But this is precisely the statement of theorem \ref{sym_poly}.

\end{proof}

\begin{remark}

In the case $\PP{\cE} = \PP_{k}^{n} = \PP(V)$, where $K$ is algebraically closed field, the constructed map is the Chow embedding:
\[
\gamma: \text{G}(1, d, n+1) \to \PP(\Sym^{d}(V)), \quad (v_1+ \cdots + v_d) \mapsto v_1 \vee \cdots \vee v_d,
\]
where $\text{G}(1, d, n+1) \cong \Sym^{d}(\PP^{n})$ is the Chow variety of $0$-cycles in $\PP^{n}$ of degree $d$.

\end{remark}

\begin{remark}

Let $\Delta: \PP{\cE} \to \Sym^{d}(\PP{\cE})$ be the diagonal embedding. Then $\varphi \circ \Delta: \PP{\cE} \to \PP{(\Sym^{d}{\cE})}$ is the $d$-fold Veronese embedding.

\end{remark}

\begin{remark}

Let $B$ be a Severi-Brauer variety of dimension $n-1$ over $K$, and let $d$ be the order of $\alpha = [B]$ in $\Br(K)$. Write $B = \PP{\cE}$, where $\cE$ is an $\alpha$-twisted vector bundle of rank $n$. Then $\Sym^{d}(\cE)$ is an ordinary bundle of rank $N = C_{d+n-1}^{d}$, so $\PP(\Sym^{d}(\cE)) \cong \PP_{K}^{N-1}$ and we have a projective embedding $\varphi: \Sym^{d}(B) \to \PP_{K}^{N-1}$.

Note that there is also an embedding $\varphi: \Sym^{d}(\PP^{n-1}) \to \PP_{K}^{N-1}$. However, if $B$ is non-trivial and $\dim{B} > 1$, then the symmetric power $\Sym^{r}(B)$ is never isomorphic to $\Sym^{r}(\PP^{n-1})$ for some integer $r \ge 0$. Indeed, it follows from \ref{theorem:automorphisms} that the $K^{sep}/K$-forms of both $\PP^{n-1}$ and $\Sym^{r}(\PP^{n-1})$ are classified by $\mathrm{H}^{1}(\Gal(K^{sep}/K), \PGL_n(K^{sep}))$. Thus we can associate with $\Sym^{r}(B)$ a unique Brauer class coinciding with the class of $B$, so $\Sym^{r}(B) \ncong \Sym^{r}(\PP^{n-1})$ (yet they are birationally isomorphic if $d = n$, since $\Sym^{n}(B)$ is rational in this case (see \cite[Theorem 3]{Kollar}).

\end{remark}




\subsection{Motivic zeta function of Severi-Brauer varieties of index 2.} Let us compute the motivic zeta function of a Severi-Brauer variety $C$ of dimension $1$ over a field, which is known to be a conic. Assume that $C$ is non-trivial Severi-Brauer variety, i. e. $C = \PP{\cE}$, where $\cE$ is a simple bundle of rank $2$ and $\End{\cE}$ is a central division algebra of degree $2$. Then $\dim(\Sym^{d}(C)) = d$, so $\varphi: \Sym^{d}(C) \to \PP(\Sym^{d}(\cE))$ is an isomorphism. 

Note that $\Sym^{d}(\cE)$ is an $\alpha^n$-twisted bundle of rank $d+1$. We have $\alpha^2 = e$ by proposition \ref{rank}, hence $\Sym^{2k}(\cE)$ is a regular bundle on $\Spec{K}$ and $\Sym^{2k}(C) \cong \PP^{2k}$. Moreover, since $\Sym^{2k+1}(\cE)$ is an $\alpha$-twisted bundle of rank $2k+2$, it follows from proposition \ref{fields} that $\Sym^{2k+1}(\cE) \cong \cE^{\oplus(k+1)}$, hence $[\Sym^{2k+1}(C)] = [C](1+\LL^2+ \cdots + \LL^{2k})$ by corollary \ref{motive_of_a_sum}. Thus we have

\[
\aligned
Z_{\text{mot}}(C, t) & = \sum_{k \ge 0} (1+\LL+\LL^2+ \cdots + \LL^{2k})
t^{2k} + \sum_{k \ge 0} \bigl([C](1+\LL^2+ \cdots + \LL^{2k}) \bigr) t^{2k+1} = \\
& = \sum_{k \ge 0} \frac{1-\LL^{2k+1}}{1- \LL}t^{2k} + \sum_{k \ge 0} \frac{[C](1-\LL^{2(k+1)})}{1- \LL^2}t^{2k+1} = \frac{1+ [C] t + \LL t^2}{(1-t^2)(1- \LL^2 t^2)}.
\endaligned
\]

\begin{proposition}

If $B$ is a Severi-Brauer variety of index 2, then $Z_{\text{mot}}(B, t)$ is rational.

\end{proposition}

\begin{proof}

This is a direct consequence of corollary \ref{rationality_of_matrix}.

\end{proof}

\subsection{Motivic zeta function of a Severi-Brauer surface} Consider the completion of $K_0(\text{Var}/K)$ at the ideal $(\LL)$:

\[
R \defeq \varprojlim K_0(\text{Var}/K)/(\LL^n)
\]

The kernel of the canonical morphism $K_0(\text{Var}/K) \to R$ is the intersection of the powers of $\LL$, but it is not known whether $\bigcap_n (\LL^n) = (0)$.

Note that the classes of $[\PP^n]$ are invertible in $R$. In particular, if $\bigcap_n (\LL^n) = (0)$, then they are not zero divisors in $K_0(\text{Var}/K)$.

\begin{theorem} \label{zetaforsurface}

Let $B$ be a Severi-Brauer surface over a field $K$ of characteristic zero. Then
\[
Z_{\text{mot}}(B, t) = t[B] \cdot \frac{1+(1+\LL^2)t+(\LL^2+\LL^4)t^3+\LL^4 t^4}{(1-t^3)(1-\LL^3 t^3)(1-\LL^6 t^3)}+\sum_{i \ge 0} [\Sym^{3i}(B)]t^{3i}
\]
in $R$. If $\bigcap_n (\LL^n) = (0)$, then the equality holds in $K_0(\text{Var}/K)$.

\end{theorem}

\begin{proof}

By proposition \ref{highsym}, we have $[\Sym^{r}(B)] \cdot[\PP^{2}] = [\Sym^{r}(\PP^{2})] \cdot [B]$
for every $r$ which is not divisible by $3$. It is easy to check that the classes $[\Sym^{3k+1}(\PP^{2})] = [\Gr_{2, 3k+3}]$ and $[\Sym^{3k+2}(\PP^{2})] = [\Gr_{2, 3k+4}]$ are all divisible by $[\PP^2]$. Assuming that  $\bigcap_n (\LL^n) = (0)$ (or arguing in $R$), we can cancel by $[\PP^2]$ and, after simple calculations, we obtain:
\[
 \sum_{i \ge 0} [\Sym^{3i+1}(B)]t^{3i+1} = [B]  \sum_{i \ge 0} \frac{[\Sym^{3i+1}(\PP^2)]}{[\PP^{2}]}
 t^{3i+1} = t[B] \cdot 
 \frac{1+(\LL^2+\LL^4)t^3}{(1-t^3)(1-\LL^3 t^3)(1-\LL^6 t^3)};
\]
\[
 \sum_{i \ge 0} [\Sym^{3i+2}(B)]t^{3i+2} = [B]  \sum_{i \ge 0} \frac{[\Sym^{3i+2}(\PP^2)]}{[\PP^{2}]}
 t^{3i+2} = t^2[B] \cdot 
 \frac{1+\LL^2+\LL^4 t^3}{(1-t^3)(1-\LL^3 t^3)(1-\LL^6 t^3)}.
\]
Adding both expressions, we get the desired formula.

\end{proof}

One can easily prove that if $r$ is coprime with $n+1$, then $\binom{n+r}{r}_q$ is divisible by $\binom{n+1}{1}_q = [n+1]_q$ as a polynomial in $q$. Let us denote this polynomial by $g_{r, n}(q)$. It follows from propositions  \ref{highsym} and \ref{Grassmannian} that for a Severi-Brauer variety of dimension $n$ and for $r$ coprime with $n+1$ one has
\[
[\Sym^{r}(B)] = \frac{[\Sym^{r}(\PP^n)]}{[\PP^n]}[B] = \frac{[\Gr_{r, r+n}]}{[\PP^n]}[B] = g_{r, n}(\LL) [B]
\]
in $R$.












\end{document}